\documentclass{amsart}
\address{\newline{\normalsize Moscow Institute of Physics and Technology, Institutskij pereulok, 9, Moskovskaya oblast', Dolgoprudnyi, Russia}
\newline{\it E-mail address}: karzhemanov.iv@mipt.ru}
\usepackage{amscd,amsthm,amsmath,amssymb}
\usepackage[matrix,arrow]{xy}

\makeatletter\@addtoreset{equation}{section}\makeatother

\makeatletter\@addtoreset{subsection}{equation}\makeatother

\newtheorem{theorem}[equation]{Theorem}
\newtheorem{prop}[equation]{Proposition}
\newtheorem{lemma}[equation]{Lemma}

\theoremstyle{remark}
\newtheorem{remark}[equation]{Remark}
\newtheorem*{notation}{Conventions}

\theoremstyle{definition}

\newtheorem{example}[equation]{Example}

\newcommand{\com}{\mathbb{C}}
\newcommand{\ra}{\mathbb{Q}}
\newcommand{\rea}{\mathbb{R}}
\newcommand{\f}{\mathbb{F}}

\newcommand{\aut}{\text{Aut}}
\newcommand{\p}{\mathbb{P}}

\newcommand{\map}{\longrightarrow}
\newcommand{\cel}{\mathbb{Z}}

\begin{document}

\pagestyle{plain}

\sloppy

\title{Rationality of an $S_6$-invariant quartic $3$-fold}

\author{Ilya Karzhemanov}

\bigskip

\begin{abstract}
We complete the study of rationality problem for hypersurfaces
$X_t\subset\p^4$ of degree $4$ invariant under the action of the
symmetric group $S_6$.
\end{abstract}

\medskip

\thanks{{\it MS 2010 classification}:  14E08, 14E30, 14M10}


\maketitle

\bigskip

\section{Introduction}
\label{section:int}

\refstepcounter{equation}
\subsection{}
\label{subsection:int-1}

Any quartic $3$-fold $X_t\subset\p^4$ with a non-trivial action of
the group $S_6$ can be given by the equations
\begin{equation}
\label{s-6-qua} \sum x_i = t\sum x_i^4 - (\sum x_i^2)^2 = 0
\end{equation}
in $\p^5$. Here the parameter $t\in\p^1$ is allowed to vary.

When $t = 2$ one gets the Burkhardt quartic whose rationality is
well-known (see e.g. \cite[5.2.7]{hun}). Similarly, $t = 4$
corresponds to the Igusa quartic, which is again rational (see
\cite[Section 3]{pro-quot}). On the other hand, it was shown in
\cite{bea} that for all other $t\ne 0,6,10/7$ the quartic $X_t$ is
non-rational.

\begin{example}
\label{example:10-7-quartic} Following \cite[Section 4]{che-sh},
let us blow up an $A_6$-orbit of $12$ lines in $\p^3$ to get a
$3$-fold that contracts, $A_6$-equivariantly, onto a quartic
threefold with $36$ nodes. It follows from {\it Remark} in
\cite{bea} that this (Todd) quartic must be $X_{10/7}$. Hence
$X_{10/7}$ is rational.
\end{example}

Thus, excluding the trivial case of $t = 0$ it remains to consider
only $X_6$, in order to completely determine the birational type
of all $S_6$-invariant quartics. Here is the result we obtain in
this paper:

\begin{theorem}
\label{theorem:main} The quartic $X := X_6$ is rational.
\end{theorem}

Theorem~\ref{theorem:main} is proved in Section~\ref{section:pro}
by, basically, running the equivariant-MMP-type of arguments as in
\cite{pro-sim-gr}. (Although the proof also uses some computations
carried in Section~\ref{section:prelim}.) Unfortunately, we were
not able to apply the results from \cite{anya}, since non-rational
$X_t$ all have defect equal $5$ (see \cite[Lemma 2]{bea}), which
seems to contradict either \cite[5.2, Lemma 8]{anya} or \cite[5.2,
Proposition 3]{anya} (compare also with \cite[Corollary 1]{anya}
and the list of cases in \cite[Main Theorem]{anya}).

The proof of Theorem~\ref{theorem:main} goes as follows. First we
show that there exists a $G$-invariant non-Cartier divisor on $X$
for some subgroup $G \subset S_6$ of order $20$ (see
Proposition~\ref{theorem:G-is-g-a}). Next we show that there
exists a $G$-equivariant small birational contraction $\phi: Y
\map X$ with terminal $\ra$-factorial $3$-fold $Y$ (see
Section~\ref{section:class}). Note that there is a $K_Y$-negative
$G$-extremal contraction $\psi: Y \map Z$ and the proof of
Theorem~\ref{theorem:main} is obtained by a detailed analysis of
possibilities for $\psi$. Namely, we show that if $\psi$ is
birational, with exceptional locus $E$, then it must be a
composition of blow-ups at smooth rational curves (see
Proposition~\ref{theorem:com-bir-con}). This argument actually
allows one to assume $\psi: Y \map Z$ to be a sequence of $G$-MMP
steps and thus reduce the proof of Theorem~\ref{theorem:main} to
the case of a $G$-Mori $3$-fold $Z$. The cases when $Z$ is a
$G$-conic fibration or a $G$-del Pezzo, as well as the
intermediate case of $E = \emptyset$, are treated in
{\ref{subsection:pro-1dkdkd}} and {\ref{subsection:pro-2}},
respectively (we get rationality of $X$ here). Thus one may assume
that $Z$ is a $G\ra$-Fano $3$-fold. Rationality in this case is
proved in {\ref{subsection:pro-3}} and {\ref{subsection:pro-4}} by
studying the $G$-action on $\text{Sing}\,Z$ and an appropriate
surface $S \in |-K_Z|$ (cf. Lemma~\ref{theorem:x-s-plt}), as well
as applying the classification results from \cite{isk-pro} to $Z$,
based on the estimates for $-K_Z^3$ (see \eqref{deg-k-z} and
\eqref{deg-k-z-1as}).

\bigskip

\begin{notation}
The ground field is $\com$ and $X$ signifies the quartic $X_6$ in
what follows. We will be using freely standard notions and facts
from \cite{isk-pro} and \cite{kol-mor} (yet we recall some of them
for convenience).
\end{notation}

\thanks{{\bf Acknowledgments.}
Some parts of the paper were prepared during my visits to AG
Laboratory at HSE (Moscow) and Miami University (US). I am
grateful to these Institutions and people there for hospitality.
Also thanks to anonymous referee for valuable comments. The work
was supported by World Premier International Research Initiative
(WPI), MEXT, Japan, and Grant-in-Aid for Scientific Research
(26887009) from Japan Mathematical Society (Kakenhi).

\bigskip

\section{Preliminaries}
\label{section:prelim}

\refstepcounter{equation}
\subsection{}
\label{subsection:prelim-1}

Recall that the singular locus $\mathrm{Sing}\,X \subset X$
consists of two $S_6$-orbits, of length $30$ and $10$,
respectively (cf. \emph{Remark} in \cite{bea}), where the first
orbit contains the point $o := [1:1:w:w:w^2:w^2], w :=
\sqrt[3]{1}$, whereas the second one contains the point $o' :=
[-1:-1:-1:1:1:1]$.

Let us consider the \emph{local class group} $\text{Cl}_{o,X}$ of
the $3$-fold $X$ at the point $o$. Namely, one identifies the
analytic germ of $X \ni o$ with the nodal quadric $X_0 := (xy =
zt) \subset \com^4$ and considers various morphisms $\mu: X_0 \map
X'$, where $X'$ is any (not necessarily normal) variety. Then for
the Weil divisor $\Delta := (x = z = 0) \subset X_0$, the group
$\text{Cl}_{o,X}$ is generated by the isomorphism classes of
sheaves $\mathcal{O}_{X_0}(\Delta)$ and $\mu^*\mathcal{O}_{X'}(H)$
for all Cartier divisors $H$ on $X'$ (note that
$\mu^*\mathcal{O}_{X'}(H)$ may no longer be a divisorial sheaf for
non-flat $\mu$), with ``$\mu^*\mathcal{O}_{X'}(H) = 0$'' being the
only relations in $\text{Cl}_{o,X}$. The group operation ``$+$''
in $\text{Cl}_{o,X}$ is induced by the tensor product of
$\mathcal{O}_X$-modules. It is immediate from the definition that
$\text{Cl}_{o,X} \simeq \cel$ (no multiple of $\Delta$ can be the
zero locus of a rational function on $X_0$).

\refstepcounter{equation}
\subsection{}
\label{subsection:prelim-2}

Next we find an effective Weil divisor on $X$ whose restriction to
$X_0$ gives a generator of $\text{Cl}_{o,X}$. Consider the
subspace $\Pi\subset\p^5$ given by equations
$$
x_0 + x_2 + x_5 = x_1 + x_3 + x_4 = 0.
$$
We have $X \cap \Pi = Q_1 + Q_2$, where the quadric
$Q_1\subset\p^3$ is given by
$$
x_0^2 + x_0x_2 + x_2^2 + w(x_1^2 + x_1x_3 + x_3^2) = 0,
$$
while the equation of $Q_2\subset\p^3$ is
$$
x_0^2 + x_0x_2 + x_2^2 - (w+1)(x_1^2 + x_1x_3 + x_3^2) = 0.
$$
Note that both $Q_i \subset X$ are non-Cartier divisors because
$Q_1 + Q_2 \sim \mathcal{O}_X(1)$.\footnote{``$\sim$'' denotes the
linear equivalence of divisors. We will also sometimes identify a
divisor with the corresponding divisorial sheaf.} Furthermore,
$Q_i$ are smooth, since they are projectively equivalent to
$$
x_0^2 + x_2^2 + x_1^2 + x_3^2 = 0.
$$
In particular, restricting to the tangent cone $X_0$ we get $Q_i
\sim \Delta$, so that $Q_i = \pm 1$ in $\text{Cl}_{o,X} = \cel$.

\refstepcounter{equation}
\subsection{}
\label{subsection:prelim-3}

Now we need some equivariant properties of $Q_i$ above. Identify
the set $\{x_2,x_0,x_4,x_3,x_1\}$ with $\{1,\ldots,5\}$ and
consider the corresponding action of the group $S_5$. Put $\tau :=
(13524)\in S_4\subset S_5$.\footnote{For the set $\{1,\ldots,n\}$,
any $n \ge 1$, the symbol $(i_1 \ldots i_n),1\le i_j\le n$,
denotes its permutation $\{i_1,\ldots,i_n\}$ (i.e. $1 \mapsto i_1$
and so on). Also, if $i_j = j$ for some $j$, then we will identify
(in the obvious way) $(i_1 \ldots i_n)$ with the permutation of
respective $(n - 1)$-element set.} Then the following (evident)
assertion holds:

\begin{lemma}
\label{theorem:g-on-q-1} $\tau^c(Q_i) \ni o$ iff $c = 0$ or $2$.
\end{lemma}

Consider $h := (23451)\in S_5$. Direct computation again gives the
following:

\begin{lemma}
\label{theorem:g-on-q-2} $h^a\tau^b(Q_i)\ni o$ iff

$(a,b)\in\{(0,0), (3,0), (4,0), (0,2), (3,3), (1,2), (4,2),
(1,1)\}$. More precisely, we have

\begin{itemize}

    \item $\tau^2(Q_i) = h^4(Q_i)\ni o$ and $\tau^2(Q_i)\ne Q_i$;

    \smallskip

    \item $h^4\tau^2(Q_i) = h^3(Q_i)\ni o$ and $h^4\tau^2(Q_i) \ne Q_i,\tau^2(Q_i)$;

    \smallskip

    \item $h\tau^2(Q_i) = Q_i$;

    \smallskip

    \item $h^3\tau^3(Q_i) = h\tau (Q_i)\ni o$ and $h^3\tau^3(Q_i) \ne
    Q_i,\tau^2(Q_i),h^4\tau^2(Q_i)$.

\end{itemize}

\end{lemma}

Let $G := \left<\tau,h\right>$ be the group generated by $\tau$
and $h$. Note that the order of $G$ is divisible by $4$ and $5$.
Then from the classification of subgroups in $S_5$ we deduce that
$G$ is the \emph{general affine group} $\mathrm{GA}(1,5)$. Note
also that $G = \f_5\rtimes\f_5^*$ for the field $\f_5$ (here
$\f_5,\f_5^*$ are the additive and multiplicative groups,
respectively).

Consider the divisor $D := \displaystyle\sum_{\gamma\in G}
\gamma(Q_1)$. By construction (the class of) $D$ belongs to the
subgroup $\mathrm{Cl}^G\,X \subseteq \mathrm{Cl}\,X$ of classes of
\emph{$G$-invariant divisors} on $X$. Note that $\mathcal{O}_X(1)
\in \mathrm{Cl}^G\,X$.

The next proposition will be crucial in the forthcoming
constructions:

\begin{prop}
\label{theorem:G-is-g-a} We have: (i) the class of $D$ in
$\mathrm{Cl}_{o,X} = \cel$ is non-zero and is divisible by $4$;
(ii) $\mathrm{rk}\,\mathrm{Cl}^G\,X
> 1$.
\end{prop}

\begin{proof}
By definition of $\tau,h$ and by Lemmas~\ref{theorem:g-on-q-1},
\ref{theorem:g-on-q-2} we have
\begin{equation}
\label{cl-o-D-e} D = \sum_{(a,b)\in\{(0,0),\ldots,(1,1)\}}
\tau^ah^b(Q_1) = 2h^4(Q_1) + 2h^3(Q_1) + 2Q_1 + 2h\tau(Q_1)
\end{equation}
in $\text{Cl}_{o,X}$, where we have identified $Q_1$ with
$\mathcal{O}_X(Q_1)$ (same for $D$). Also, since
$h^3(Q_1),h^4(Q_1) \ni o$, both $h^3,h^4$ act on
$\text{Cl}_{o,X}$. Indeed, $h^3(Q_1)$ and $h^4(Q_1)$ differ from
(a power of) $Q_1$ by a suitable $\mu^*H$ (cf.
{\ref{subsection:prelim-1}}).

Since $h^3 = (h^4)^2$, we get $h^3(Q_1) = 1 = (h^3)^3(Q_1) =
h^4(Q_1)$, and hence $D = 4$ or $8$. This means in particular that
the product of $\mathcal{O}_X$-modules
$$
\mathcal{I} := \prod_{\gamma\in G}\mathcal{O}_X(-\gamma(Q_1)),
$$
identified with $D$ as an element in $\text{Cl}_{o,X}$, is not
invertible (otherwise $D$ will be zero).

Take a $G$-equivariant resolution of singularities $r: W \map X$.
Then the sheaf $r^*\mathcal{I}$ becomes invertible. Let
$D_{\mathcal{I}}$ be divisor on $W$ such that $r^*\mathcal{I} =
\mathcal{O}_W(D_{\mathcal{I}})$. Note that $D_{\mathcal{I}}$ is
\emph{effective} by construction of $\mathcal{I}$. Write
$D_{\mathcal{I}} = D' + \Xi$, where $D'$, $\Xi$ are effective
without common irreducible components, and $\Xi$ consists of
$r$-exceptional divisors.

\begin{lemma}
\label{theorem:G-is-g-a-lemma} $D'$ is \emph{not} $r$-relatively
trivial.
\end{lemma}

\begin{proof}
Indeed, otherwise $\mathcal{I}$ will be equal to
$\mu_*\mathcal{O}_{X'}(H)$ as in {\ref{subsection:prelim-1}} by
the projection formula, which is impossible as $D \ne 0$ in
$\text{Cl}_{o,X}$.
\end{proof}

Apply relative $G$-equivariant MMP to $W$ (cf. \cite[9.1]{sho}).
This is a sequence of $G$-equivariant birational contractions $W_i
\map W_{i+1}$ over $X$, starting with $W_0 := W$, which contracts
the above divisor $\Xi$ and results in a \emph{small} (no
divisorial exceptional part) $G$-equivariant contraction $\phi: Y
\map X$. Furthermore, it follows from
Lemma~\ref{theorem:G-is-g-a-lemma} that the proper (birational)
transform of $D_{\mathcal{I}}$ on $Y$ is a \emph{$\phi$-relatively
non-trivial} $G$-invariant Cartier divisor $D_Y$, which coincides
with the proper transform of $D'$.

By construction $\phi(D_Y)$ and $\mathcal{O}_X(1)$ are
non-proportional Weil divisors on $X$. This implies that
$\mathrm{rk}\,\mathrm{Cl}^G\,X > 1$ and completes the proof of
Proposition~\ref{theorem:G-is-g-a}.
\end{proof}

\begin{remark}
\label{remark:commet-fkfk-1} The present definition of
$\text{Cl}_{o,X}$ differs from the usual (algebraic) one that is
via the direct limit of groups $\text{Cl}\,U\slash\text{Pic}\,U$
over all Zariski open subsets $U \ni o$ in $X$ (cf.
\cite[Definition 7]{kollar-pic}). A priori there is no natural
isomorphism of the latter with $\text{Cl}_{o,X}$. At the same
time, we have used the fact that $0 \ne D \in \text{Cl}_{o,X}$ in
order to construct $Y$ above, thus proving the existence of
\emph{some} $G$-invariant non-Cartier divisor on $X$.
\end{remark}

\begin{remark}
\label{remark:commet-fkfk} Let us also comment a bit on the
preceding technicalities. The estimate
$\mathrm{rk}\,\mathrm{Cl}^G\,X
> 1$ will be used, essentially, in the arguments of
Section~\ref{section:pro} below. But we stress that finding a
particular ($G$-invariant) non-Cartier divisor on $X$ is not that
easy. For instance, in order to find one, we may proceed as
follows. Consider the projection $\pi: X \dashrightarrow \p^3$
from either $o$ or $o'$, which is a rational map of degree $2$,
and let $R \subset X$ be the ramification divisor of $\pi$. Then
one could try to search a non-Cartier irreducible component in $R$
-- for instance the closure $R_0$ of the locus on which $\pi$ is
not finite. However, a direct computation shows that $R$ is
reduced, irreducible and $\dim R_0 \le 1$, with $R \sim
\mathcal{O}_X(6)$ being Cartier by construction.
\end{remark}

\bigskip

\section{Auxiliary results}
\label{section:class}

\refstepcounter{equation}
\subsection{}
\label{subsection:class-3}

Fix $\phi: Y \map X$ as in the proof of
Proposition~\ref{theorem:G-is-g-a}. This is a particular instance
of \emph{terminal $G\ra$-factorial modification} (of $X$). Namely,
in addition to $\phi$ being $G$-equivariant and small, the
$3$-fold $Y$ is also terminal and $G\ra$-factorial, i.e. every
$G$-invariant divisor on $Y$ is $\ra$-Cartier.

\begin{lemma}
\label{theorem:y-is-gor} $Y$ is Gorenstein.
\end{lemma}

\begin{proof}
This follows from the relation $\phi^*\omega_X = \omega_Y$, the
fact that $\phi$ is small, and the freeness of $|-K_X|$ for $-K_X
\sim \mathcal{O}_X(1)$.
\end{proof}

\begin{lemma}
\label{theorem:no-of-sings-y} One can choose $Y$ to be such that
$\mathrm{Sing}\,Y = \emptyset$ or $G\cdot o'$.
\end{lemma}

\begin{proof}
Indeed, the divisor $D$ from {\ref{subsection:prelim-3}} contains
the point $o$ and the morphism $\phi$ makes $X$ $G\ra$-factorial
near $o$, which means that one may take $\phi$ to resolve the
singularities in $G\cdot o \subset D$ (simply run the
$G\ra$-factorialization procedure as in the proof of
Proposition~\ref{theorem:G-is-g-a}).

The complement $\Sigma := \mathcal{O}\setminus G \cdot o$, where
$\mathcal{O}$ is the longest $S_6$-orbit in $\text{Sing}\,X$, is
also a $G$-orbit (of length $10$). Furthermore, we have $s(o) \ne
o\in\Sigma$ for $s := (21)\in S_5$, and so the arguments in the
proof of Proposition~\ref{theorem:G-is-g-a}, with $s(Q_1) = Q_1$,
apply to show that $X$ not $G\ra$-factorial near $\Sigma$ as well.
Hence we may assume that $\phi$ also resolves the singularities at
$\Sigma$.

Finally, $\phi$ either resolves or not the singularities in
$G\cdot o'$, depending on whether there is a $G$-invariant
non-Cartier divisor passing through $o'$ or there is no such.
\end{proof}

\refstepcounter{equation}
\subsection{}
\label{subsection:class-4}

From now on we will assume that $Y$ is as in
Lemma~\ref{theorem:no-of-sings-y}.

\begin{lemma}
\label{theorem:lin-indep-e-i} $Y$ is $\ra$-factorial with
$\mathrm{rk}\,\mathrm{Pic}\,Y = 11$.
\end{lemma}

\begin{proof}
Note that $\f_5 = \left<h\right>$ is the unique normal subgroup in
$G = \f_5\rtimes\f_5^*$. Then we have $Q_i\not\sim h(Q_i)$.
Indeed, otherwise $D \sim \displaystyle
5\sum_{\gamma\in\left<\tau\right>}\gamma(Q_i)$, with $D$ as above.
But in this case $D = 5(Q_1 + \tau^2(Q_1))$ in $\text{Cl}_{o,X}$
(see Lemma~\ref{theorem:g-on-q-2}), which is either $0$ or $10$,
thus contradicting Proposition~\ref{theorem:G-is-g-a}, (i).

Further, since $D$ is a $G$-orbit of $Q_1$, all of its components
are linearly independent in $\mathrm{Cl}\,X\otimes\rea$. Indeed,
otherwise we get $\sum \gamma(Q_1) = 0$, which is an absurd. This,
together with a computation of the defect in \cite{bea}, yields
$\text{rk}\,\mathrm{Cl}\,X = 11$ for the class group
$\mathrm{Cl}\,X$ being generated by $K_X$ and by the components of
$D$ (the number of these components is $10$ because \linebreak
$Q_1\not\sim h(Q_1)$).

Similarly, we find that $\mathrm{Cl}\,Y$ is generated by $K_Y$ and
by the components of $\phi_*^{-1}D$,\footnote{``$\phi_*^{-1}$''
denotes the proper (birational) transform on $Y$ of a divisor on
$X$ with respect to $\phi$.} all being Cartier according to
Lemma~\ref{theorem:no-of-sings-y} and the fact that $D\not\ni o'$.
Thus we get $\mathrm{Cl}\,Y = \mathrm{Pic}\,Y$ and the claim
follows.
\end{proof}

Recall that $Y$ is terminal and Gorenstein (see
Lemma~\ref{theorem:y-is-gor}). This allows for a $K_Y$-negative
$G$-extremal contraction $\psi: Y \map Z$ (cf. \cite{pr-sh}). For
the rest of this section, we will assume that $\psi$ is
\emph{birational}, with exceptional locus $E$. Let $E_i$ be the
irreducible components of $E$. Here are two auxiliary results
about these:

\begin{lemma}
\label{theorem:psi-bir-1} Suppose $\dim \psi(E) = 0$. Then $E$ is
a \emph{disjoint} union of $E_i$.
\end{lemma}

\begin{proof}
Since the divisor $-K_Y$ is nef and big, it follows from
Lemmas~\ref{theorem:y-is-gor}, \ref{theorem:lin-indep-e-i} and
\cite{pr-sh} that the Mori cone $\overline{NE}(Y)$ is polyhedral,
is spanned by extremal rays, and every extremal ray on $Y$ is
contractible. This implies that some family of curves in each
$E_i$ generates an extremal ray because there are no small
$K_Y$-negative extremal contractions on $Y$ (see \cite[Lemma
5.1]{kaw} and \cite{cut}). In particular, $E_i$ do not intersect,
as $\dim\psi(E) = 0$ by assumption.
\end{proof}

\begin{lemma}
\label{theorem:y-is-q-f} Suppose $\dim \psi(E) = 0$. Then every
surface $E_i$ is \emph{not} preserved by the subgroup
$\left<h\right>\subset G$ (cf. {\ref{subsection:prelim-3}}).
\end{lemma}

\begin{proof}
Note that $\text{Cl}\,X\simeq\text{Cl}\,Y$ as $G$-modules. This
induces a natural $G$-action on the cone $\overline{NE}(Y)$.
Consider the $G$-extremal face of $\overline{NE}(Y)$ given by
$\psi$. By Lemma~\ref{theorem:psi-bir-1} this is spanned by a
$G$-orbit of some $K_Y$-negative contractible extremal rays $R_i$
corresponding to $E_i$.

Let some $E_j$ be $\left<h\right>$-invariant. We have $\phi(E_j) =
\p^1\times\p^1$, a quadratic cone, or $\p^2$ by \cite[Lemma
5.1]{kaw} and \cite{cut}. In particular, there is a projective
subspace $\p^3\subset\p^4\supset X$ (with $\phi(E_j) \subseteq X
\cap \p^3$), invariant under $\f_5 = \left<h\right>$. Recall that
$h = (23451)$ permutes $x_0,x_2,x_1,x_3,x_4$. Hence the equation
of $\p^3$ is $\displaystyle\sum_{i=0}^4 x_i = 0$. This implies
that $X\cap\p^3\cap\text{Sing}\,X = \emptyset$ and so $\phi(E_j)$
is Cartier. But the latter is impossible for otherwise $\phi(E_j)$
would intersect all the curves on $X$ negatively.
\end{proof}

We now prove another important result to be used in the proof of
Theorem~\ref{theorem:main}:

\begin{prop}
\label{theorem:psi-bir} Let $\psi$ be as above. Then $\psi(E)$ is
a curve.
\end{prop}

\begin{proof}
Assume the contrary. Then it follows from
Lemma~\ref{theorem:psi-bir-1}, \ref{theorem:y-is-q-f} that all
$E_i$ are linearly independent in $\text{Pic}\,Y\otimes\rea$, and
together with $K_Y$ they generate $\text{Pic}\,Y$ (cf. the proof
of Lemma~\ref{theorem:lin-indep-e-i}).

Let $R_i \subset \overline{NE}(Y)$ be the extremal ray
corresponding to $E_i$. We have $E_i\cdot C\ge 0$ for all $i$ and
any $K_Y$-trivial curve $C\subset Y$ because otherwise the class
of $C$ belongs to $R_i$ (which is $K_Y$-negative by construction).
In particular, there is such $C$ that any other $K_Y$-trivial
curve on $Y$ is numerically equivalent to $C + \sum a_iR_i$ for
some $a_i\geq 0$, which forces all $a_i = 0$. This implies that
all $K_Y$-trivial curves on $Y$ are numerically proportional and
so $E_i \cdot C
> 0$.

Hence every surface $\phi(E_i) \subseteq X \cap \p^3$ (of degree
$(K_Y)^2\cdot E_i\leq 2$) contains a $G$-orbit of length at least
$30$ (see Lemma~\ref{theorem:no-of-sings-y}). Thus we obtain that
$\phi(E_i)$, together with $E_i$, are all
$\left<h\right>$-invariant because there are no $G$-invariant
curves in $\mathbb{P}^3\cap S_1\cap S_2$ for two different
surfaces $S_i$ of degree $\le 2$ containing common $G$-orbit of
length $30$ (cf. Lemma~\ref{theorem:not-gl-3} below). However,
this $\left<h\right>$-invariance of $\phi(E_i)$ contradicts
Lemma~\ref{theorem:y-is-q-f}, and
Proposition~\ref{theorem:psi-bir} is proved.
\end{proof}

We conclude by the following simple (although useful in what
follows) observation:

\begin{lemma}
\label{theorem:not-gl-3} $G\not\subset\mathrm{GL}(3,\com)$.
\end{lemma}

\begin{proof}
The group $G$ has only one $4$-dimensional and four
$1$-dimensional irreducible representations. The claim follows by
decomposing $\com^3$ into the direct sum of irreducible
$G$-modules.
\end{proof}

\bigskip

\section{Proof of Theorem~\ref{theorem:main}}
\label{section:pro}

\refstepcounter{equation}
\subsection{}
\label{subsection:pro-1}

We retain the notation of Section~\ref{section:class}. Consider
some $K_Y$-negative $G$-extremal contraction $\psi: Y \map Z$. Let
us assume for a moment that $\psi$ is birational with exceptional
locus $E$. Recall that $E$ is a union of (generically) ruled
surfaces $E_i$ contracted by $\psi$ onto some curves (see
Proposition~\ref{theorem:psi-bir}).

\begin{lemma}
\label{theorem:rk-pic-y-1} We have $E\cap\mathrm{Sing}\,Y =
\emptyset$.
\end{lemma}

\begin{proof}
Over generic point of $\psi(E_i)$ morphism $\psi$ coincides with
the blow-up of a curve (see \cite{cut}). Then for any ruling $C
\subset E_i$ contracted by $\psi$ we have $K_Y \cdot C = -1$.
Hence the surfaces $\phi(E_i)\subset X$ are swept out by the lines
$\phi(C)$.

Note that $C$ corresponds to a contractible extremal face of
$\overline{NE}(Y)$ (cf. the proof of
Lemma~\ref{theorem:psi-bir-1}). In particular, one may assume that
$C$ generates a $K_Y$-negative extremal ray, which shows that $C$
is Cartier on $E_i$ because all scheme fibers of
$\psi\big\vert_{E_i}$ are smooth (lines) and $C$ varies in a flat
family.

Further, it follows from Lemmas~\ref{theorem:y-is-gor},
\ref{theorem:lin-indep-e-i} and \cite[Lemma 5.1]{kaw} that all
divisors $E_i$ are Cartier. Now, if
$E_i\cap\mathrm{Sing}\,Y\ne\emptyset$, then $\phi(C)$ is a
singular curve for some $C$ as above, which is impossible. Hence
$E_i\cap\mathrm{Sing}\,Y = \emptyset$ and so
$E\cap\mathrm{Sing}\,Y = \emptyset$.
\end{proof}

\begin{remark}
\label{remark:e-i-is-ruled} We have $h^{1,2} = 0$ for a resolution
of $Y$ according to {\it Remark} in \cite{bea}. Then it follows
from \cite{cut} and Lemma~\ref{theorem:rk-pic-y-1} that
$\psi(E_i)=\p^1$ for all $i$.
\end{remark}

\begin{lemma}
\label{theorem:can-class-form} We have $K_Y = \psi^*K_Z + E$
(hence $Z$ is Gorenstein), $K_Y \cdot C = -1$ for any ruling $C
\subset E_i$ contracted by $\psi$, and $Z$ is smooth near
$\psi(E)$.
\end{lemma}

\begin{proof}
One obtains the first two identities by exactly the same argument
as in the proof of Lemma~\ref{theorem:rk-pic-y-1}. Further, since
the linear system $|-K_Y| = |\phi^*(-K_X)|$ is basepoint-free,
generic surface $S \in |-K_Y|$ passing through a given point from
$Y \setminus \mathrm{Sing}\,Y $ is smooth. Then, for $S \cdot C =
1$ we find that the surface $\psi(S) \in |-K_Z|$ is smooth as
well, hence $Z$ is smooth near $\psi(E)$.
\end{proof}

\begin{lemma}
\label{theorem:e-ex-asa} $E$ can not consist of only one
(connected) surface.
\end{lemma}

\begin{proof}
Assume the contrary. Note that $Y$ contains the $G$-orbit of $20$
curves $C_j$ contracted by $\phi$ (see
Lemma~\ref{theorem:no-of-sings-y}). In particular, $G$ induces a
non-trivial action on the set of these $C_j$, which implies that
$(E = E_i)\cap C_j\ne\emptyset$ for all $j$. This yields a
faithful $G$-action on the base of the ruled surface $E$. Hence we
get $G\subset\text{PGL}(2,\com)$ (see
Remark~\ref{remark:e-i-is-ruled}). On the other hand, we have
$G\not\subset A_5$ (see Lemma~\ref{theorem:not-gl-3}), a
contradiction.
\end{proof}

Here is a refinement of Lemma~\ref{theorem:e-ex-asa}:

\begin{lemma}
\label{theorem:e-is-orb} $E$ is a \emph{disjoint} union of
$G$-orbits, length $\ge 2$, corresponding to extremal faces of
$\overline{NE}(Y)$.
\end{lemma}

\begin{proof}
Let $E,\tilde{E}$ be two $\psi$-exceptional orbits in question.
Choose some connected components $E_j\subset E,\tilde{E}_j\subset
\tilde{E}$ and suppose they intersect. Both $E_j,\tilde{E}_j$ are
ruled surfaces contracted by the blow-downs, one for each surface
(cf. the proof of Lemma~\ref{theorem:rk-pic-y-1}).

Let $\psi_j: Y \map Y_j$ be the contraction of $E_j$. Then, given
that $E_j\cap\tilde{E}_j\ne\emptyset$, there is a
$\psi$-exceptional curve $C\subset\tilde{E}_j$ such that $E_j\cdot
C\ge 0$. On the other hand, we have $K_Y = \psi_j^*K_{Y_j}+E_j$
and $K_{Y_j}\cdot\psi_j(C) = -1$ (see
Lemma~\ref{theorem:can-class-form}), which gives either $K_Y \cdot
C = -1$ or $K_Y \cdot C = 0$ (recall that $-K_Y$ is nef). The
latter case is an absurd by construction of $\psi$. In the former
case, we get $E_j \cdot C=0$ and so
$\psi_*(E_j\cap\tilde{E}_j)=\psi_*C=0$, which is impossible for
the ruled surfaces $E_j\ne \tilde{E}_j$, since then $0 = E_i \cdot
C = (C^2) < 0$ on $E_i$, a contradiction.
\end{proof}

We collect the preceding results into the following:

\begin{prop}
\label{theorem:com-bir-con} Let $\psi$ be the result of running a
$G$-MMP on $Y$. Then $\psi$ is a birational contraction that maps
its exceptional loci onto $1$-dimensional centers and all the
intermediate $3$-folds are smooth near these centers. In
particular, all these $3$-folds are $\ra$-factorial, Gorenstein
and terminal, with nef and big $-K$, and $\psi$ is composed of
blow-ups at smooth rational curves.
\end{prop}

\begin{proof}
It follows from Lemmas~\ref{theorem:y-is-gor},
\ref{theorem:lin-indep-e-i}, \ref{theorem:rk-pic-y-1},
\ref{theorem:can-class-form} and \cite[Proposition-definition 4.5,
Corollary 4.9]{pro-deg} that each step of $\psi$ produces a
$\ra$-factorial, Gorenstein terminal $3$-fold, with a $G$-action
and nef and big $-K$, unless all exceptional $E_i$ have
anticanonical degree $\le 2$ on this step. In the latter case,
arguing as in the proof of Lemma~\ref{theorem:e-is-orb} one
computes that the proper transforms of $E_i$ on $Y$, hence on $X$
as well, will also have degree $\le 2$. This yields an
$\left<h\right>$-invariant quadric on $X$ (cf. the proof of
Proposition~\ref{theorem:psi-bir}) and a contradiction with
Lemma~\ref{theorem:not-gl-3}.

Similarly, whenever $E_i$ is a quadric or $\p^2$, contracted to a
point in both cases (cf. \cite[Proposition-definition 4.5,
Corollary 4.9]{pro-deg}), we get contradiction with
Lemma~\ref{theorem:not-gl-3}. Thus on each step $\psi$ can
contract $E_i$ only to curves. The final assertion of
Proposition~\ref{theorem:com-bir-con} follows from \cite{cut} and
Remark~\ref{remark:e-i-is-ruled}.
\end{proof}

\refstepcounter{equation}
\subsection{}
\label{subsection:pro-1dkdkd}

Let us treat the intermediate case when $E = \emptyset$ (i.e.
$\phi$ is non-birational). Recall that $Y$ contains the $G$-orbit
of $20$ curves $C_j$ contracted by $\phi$ (see
Lemma~\ref{theorem:no-of-sings-y}). Then we get
$\text{rk}\,\text{Pic}^G\,Y = 2$ and $\overline{NE}(Y)$ is
generated by (the $G$-orbits of) the classes of $C_j$ and an
extremal ray corresponding to some $G$-Mori fibration $\varphi:
Y\map S$. Note that $\dim S
> 0$ by construction.

\begin{lemma}
\label{theorem:e-ex-lemma} Let $\dim S = 1$. Then $Y$ is minimal
over $S$ unless it is rational.
\end{lemma}

\begin{proof}
Suppose there is a surface $\Xi$ which is exceptional for some
$K_Y$-negative extremal contraction on $Y/S$. Then $\Xi$
necessarily contains one of $C_j$. Indeed, otherwise $\Xi$
intersects all curves on $Y$ non-negatively by the structure of
$\overline{NE}(Y)$, which is impossible. In particular, we find
that $\Xi$ must be a minimal ruled surface (same argument as in
the proof of Lemma~\ref{theorem:rk-pic-y-1}), with the negative
section equal some $C_j$.

We may assume $K^2_{Y_{\eta}}\le 4$ for generic fiber $Y_{\eta}$
of $\varphi$ -- otherwise $Y$ is rational (see \cite{gr-har-st},
\cite{manin-cub-f}). Moreover, we have $K^2_{Y_{\eta}}\ne 1$,
since otherwise the group $G\subseteq\aut(Y_{\eta})$ must act
\emph{faithfully} on elliptic curves from $|-K_{\eta}|$, which is
impossible by Lemma~\ref{theorem:not-gl-3}. One also has
$K^2_{Y_{\eta}}\ne 2$ because the order of the group of
automorphisms of del Pezzo surfaces of degree $2$ is not divisible
by $5$ (see e.g. \cite[Table 8.9]{dolgachev}).

Further, if $K^2_{Y_{\eta}} = 4$, then contracting $\Xi$ we arrive
at a del Pezzo fibration of degree $5$, so that $Y$ is rational.

Now, if $K^2_{Y_{\eta}} = 3$, then all smooth fibers of $\varphi$
are isomorphic and have $\aut\,Y_{\eta} = S_5$ (see \cite[Table
9.6]{dolgachev}). Away from the singular fibers $\varphi$ defines
a locally trivial (in analytic topology) fibration of smooth cubic
surfaces $Y_{\eta}$. Two charts, $Y_{\eta}\times S'$ and
$Y_{\eta}\times S''$, say (for some analytic subsets
$S',S''\subseteq S$), are glued together via an automorphism
$t\in\aut\,Y_{\eta}$, which preserves the elements in the
$G$-orbit of $\Xi$ and satisfies $tGt^{-1} = G$. Since $G$ is not
a normal subgroup in $S_5$, one gets $t\in G$, and the latter is
impossible once $t\ne 1$ -- by the way $G$ acts on $\Xi$ and
$C_j$. Thus we have $t = 1$ and $\varphi$ induces a locally
trivial fibration in the Zariski topology, so that $Y$ is
rational, and the proof is complete.
\end{proof}

Let $\dim S = 1$ and observe that the subgroup $\left<h\right>
\subset G$ must act faithfully on $\mathrm{Pic}\,Y$ in this case.
Indeed, otherwise $Q_i\sim h^a(Q_i)$ for all $a,i$, which implies
that $Q_i$ contains the orbit $\left<h\right>\cdot o$, a
contradiction (cf. {\ref{subsection:prelim-3}}). Further, from
Lemma~\ref{theorem:e-ex-lemma} we deduce that either
$\text{Pic}\,Y = \cel^2$ (which contradicts
Lemma~\ref{theorem:lin-indep-e-i}), or $\varphi$ contains a fiber
with $\ge 5$ irreducible components (interchanged by
$\left<h\right>$). In the latter case, we get $K^2_{Y_{\eta}}\ge
5$ for generic fiber $Y_{\eta}$, and rationality of $Y$ follows
from \cite{gr-har-st}, \cite{manin-cub-f}.

Finally, one treats the case when $\varphi$ is a $G$-conic bundle
exactly as in the proof of Lemma~\ref{theorem:q-not-c-b} below,
which yields rationality of $Y$ (and $X$) whenever $E =
\emptyset$.

\refstepcounter{equation}
\subsection{}
\label{subsection:pro-2}

According to {\ref{subsection:pro-1dkdkd}} we may assume from now
on that $E \ne \emptyset$ and $\psi: Y \map Z$ is as in
Proposition~\ref{theorem:com-bir-con}. Then the $3$-fold $Z$ is
$\ra$-factorial, Gorenstein and terminal, with nef and big $-K_Z$.
Furthermore, $Z$ is either a $G$-equivariant del Pezzo fibration,
a $G$-conic bundle or a $G\ra$-Fano $3$-fold.

\begin{lemma}
\label{theorem:q-j-not-e} We have $\phi_*^{-1}Q_j\not\subset E$
for some $j$.
\end{lemma}

\begin{proof}
Note that $\psi_*K_Y = K_Z$ because $Z$ has rational
singularities. This gives the claim as $-K_Y = \phi_*^{-1}Q_1 +
\phi_*^{-1}Q_2$.
\end{proof}

\begin{lemma}
\label{theorem:q-not-c-b} $Z$ is not a $G$-conic bundle.
\end{lemma}

\begin{proof}
Suppose we are given a $G$-conic bundle structure on $Z$ with
generic fiber $C = \p^1$. Then, if $\phi_*^{-1}Q_1\not\subset E$,
say (see Lemma~\ref{theorem:q-j-not-e}), it follows from the
definition of $Q_i$ and $G$ in {\ref{subsection:prelim-3}} that
the $G$-orbit of $Q_1$ (hence also of $\phi_*^{-1}Q_1$) has length
$\ge 10$ (cf. the proof of Lemma~\ref{theorem:lin-indep-e-i}).
This yields a faithful $G$-action on $C$ which in turn contradicts
Lemma~\ref{theorem:not-gl-3}.
\end{proof}

\begin{lemma}
\label{theorem:q-not-in-e} $Z$ is not a $G$-del Pezzo fibration
unless $Z$ is rational.
\end{lemma}

\begin{proof}
Argue exactly as in the proof of Lemma~\ref{theorem:e-ex-lemma}.
\end{proof}

\refstepcounter{equation}
\subsection{}
\label{subsection:pro-3}

According to Lemmas~\ref{theorem:q-not-c-b} and
\ref{theorem:q-not-in-e} we may assume from now on that $Z$ is a
$G\ra$-Fano $3$-fold. Note that any two components of the
exceptional locus $E$ of $\psi$ can intersect only along the
fibers (cf. the proof of Lemma~\ref{theorem:e-is-orb}). Also
recall that $\mathrm{rk}\,\mathrm{Pic}\,Y = 11$ by
Lemma~\ref{theorem:lin-indep-e-i} and the subgroup $\left<h\right>
\subset G$ acts faithfully on $\mathrm{Pic}\,Y$. Then it follows
from Remark~\ref{remark:e-i-is-ruled} and
Lemmas~\ref{theorem:can-class-form}, \ref{theorem:e-is-orb} that
either
\begin{equation}
\label{deg-k-z} -K_Z^3 = 4 + 2k(-K_Z \cdot \p^1 + 1)
\end{equation}
for some even $k \le 10$ when $E_i \subset E$ do not intersect, or
\begin{equation}
\label{deg-k-z-1as} -K_Z^3 = 4 + 20(-K_Z \cdot \p^1 + 1 - k')
\end{equation}
for some $k' \le -K_Z \cdot \p^1$ when some $E_i \subset E$
intersect.

\begin{lemma}
\label{theorem:bs-is-emp} The linear system $|-K_Z|$ is
basepoint-free.

\end{lemma}

\begin{proof}
Assume the contrary. Then it follows from \cite{jah-rad} that $Z$
is a $G$-equivariant double cover of the cone over a ruled surface
(note that according to \eqref{deg-k-z} and \eqref{deg-k-z-1as}
$-K_Z^3\ge 12$ is divisible by $4$). This easily gives
$G\subset\text{PGL}(2,\com)$ and contradiction with
Lemma~\ref{theorem:not-gl-3}.
\end{proof}

\begin{lemma}
\label{theorem:k-z-is-emb} The morphism defined by $|-K_Z|$ is an
embedding.
\end{lemma}

\begin{proof}
Assume the contrary. Then it follows from \cite[Theorem 1.5]{cps}
that $Z$ is a $G$-equivariant double cover of either a rational
scroll or the cone over a ruled surface. In both cases, arguing
similarly as in the proof of Lemma~\ref{theorem:bs-is-emp}, one
gets contradiction.
\end{proof}

Lemmas~\ref{theorem:bs-is-emp} and \ref{theorem:k-z-is-emb} allow
one to identify $Z$ with its anticanonical model
$Z_{2g-2}\subset\p^{g+1}$ (here $g := -K_Z^3/2 + 1$ is the genus
of $Z$).

\begin{lemma}
\label{theorem:no-g-p-1} There are no $G$-fixed points and
$G$-invariant smooth rational curves on $Z$.
\end{lemma}

\begin{proof}
Firstly, since $G\not\subset\text{GL}(3,\com)$, the group $G$ acts
on $Z$ without smooth fixed points. Also, since $Z$ is
$G$-isomorphic to $X$ near $\mathrm{Sing}\,Z$ by construction, we
obtain that $G$ does not have fixed points on $Z$ at all.

Further, if $\p^1 \subset Z$ is $G$-invariant, then the action
$G\circlearrowright\p^1$ is cyclic, which gives a $G$-fixed point
on $\p^1$, a contradiction.
\end{proof}

\begin{lemma}
\label{theorem:sing-or-rat} $Z$ is singular unless it is rational.
\end{lemma}

\begin{proof}
Suppose that $Z$ is smooth. Then rationality of $Z$ follows from
the fact that $h^{1,2}(Z)=0$ (see
Remark~\ref{remark:e-i-is-ruled}) and \cite[\S\S 12.2 --
12.6]{isk-pro}.
\end{proof}

According to Lemmas~\ref{theorem:sing-or-rat},
\ref{theorem:no-of-sings-y} and
Proposition~\ref{theorem:com-bir-con} we may reduce to the case
when $|\mathrm{Sing}\,Z| = |\mathrm{Sing}\,Y| = 10$, with the
locus $\mathrm{Sing}\,Z$ being some $G$-orbit.

\begin{lemma}
\label{theorem:x-s-plt} Let $S \in |-K_Z|$ be a $G$-invariant
hyperplane section such that $S\cap\mathrm{Sing}\,Z\ne\emptyset$.
Then the pair $(Z,S)$ is plt.
\end{lemma}

\begin{proof}
Lemma~\ref{theorem:no-g-p-1} and the proof of \cite[Lemma
4.6]{pro-sim-gr} show that the pair $(Z,S)$ is log canonical.
Moreover, if $(Z,S)$ is not plt, then the same argument as in {\it
loc.cit} reduces the claim to the case when $S$ is a ruled surface
over an elliptic curve, say $B$. On the other hand, since
$|S\cap\mathrm{Sing}\,Z| = 10$, we get either
$G\subset\text{PGL}(2,\com)$ or a faithful $G$-action on $B$, a
contradiction.
\end{proof}

We are ready to prove the following:

\begin{prop}
\label{theorem:g-is-small} $g \le 9$.
\end{prop}

\begin{proof}
Let $g > 9$. Note that the linear span of any $G$-orbit in
$\mathrm{Sing}\,Z$ has dimension $\le 9$. Hence we can consider a
$G$-invariant hyperplane section $S \in |-K_Z|$ such that
$S\cap\mathrm{Sing}\,Z\ne\emptyset$.

It follows from Lemma~\ref{theorem:x-s-plt} and \cite[Corollary
3.8]{sho} that $S$ is either normal or reducible. But in the
latter case, $-K_Z \sim [\text{some disconnected surface}]$
because $(Z,S)$ is plt, which is impossible.

Thus the surface $S$ is normal with at most canonical
singularities. Let us identify $S$ with its $G$-equivariant
minimal resolution. In particular, we may assume that $S$ contains
a $G$-invariant collection of \emph{disjoint} $(-2)$-curves
$C_i,1\le i\le 10$.

From $G\subseteq\text{Aut}\,S$ one obtains a $G$-action on the
space $H^{2,0}(S) = \com[\omega_S]$. In particular, the subgroup
$\left<\tau^2\right>\subset G$ preserves the $2$-form $\omega_S$,
which implies that the quotient $S_{\tau} :=
S\slash\left<\tau^2\right>$ has at most canonical singularities.
Note also that $\tau^2(C_i) = C_i$ and $h(C_i)\ne C_i$ for all
$i$.

Let $\tilde{C}_i$ be the image of $C_i$ on $S_{\tau}$. We have
$|\tilde{C}_i\cap\mathrm{Sing}\,S_{\tau}| = 2$ for all $i$ because
$(\tilde{C}_i^2) = -1$ by the projection formula. Then for the
minimal resolution $S'_{\tau}$ of $S_{\tau}$ we obtain that
$S'_{\tau}$ contains $\ge 20$ disjoint $(-2)$-curves. This
contradicts $h^{1,1}(S'_{\tau}) = 20$ and finishes the proof of
Proposition~\ref{theorem:g-is-small}.
\end{proof}

According to Proposition~\ref{theorem:g-is-small} and
\eqref{deg-k-z}, \eqref{deg-k-z-1as} we may assume that
$-K_Z^3\in\{12,16\}$. (Note that the case $k = 10$ yields
$\mathrm{rk}\,\mathrm{Pic}\,Z = 1$ and can be excluded exactly as
in the proof of Proposition~\ref{theorem:pic-z-or-rat-1} below.)

\begin{remark}
\label{remark:k-not-16} Suppose $Z = Z_{16}\subset\p^{10}$. Then,
since the projective $G$-action is induced from the linear one on
$\com^{11} = H^0(Z,-K_Z)$, one gets a pencil on $Z$ consisting of
$G$-invariant hyperplane sections. In particular, there is such
$S$ intersecting $\text{Sing}\,Z$, so that the arguments in the
proof of Proposition~\ref{theorem:g-is-small} apply to exclude the
case $-K_Z^3=16$.
\end{remark}

\refstepcounter{equation}
\subsection{}
\label{subsection:pro-4}

It follows from Lemma~\ref{theorem:com-bir-con} and
\cite{namikawa} that there is a $1$-parameter family $s:
\mathcal{Z}\map\Delta$ over a small disk $\Delta\subset\com$ of
smooth Fano $3$-folds $Z_t,t\ne 0$, deforming to $Z_0 = Z$. Since
$H^i(Z_t,nK_{Z_t})=0$ for all $n\le 0,i>1$ and $t$, we deduce that
the sheaf $s_*(-K_{\mathcal{Z}})$ is locally free.

Further, the cone $\overline{NE}(Z)$ is polyhedral, with
contractible extremal rays (cf. the proof of
Lemma~\ref{theorem:psi-bir-1}). Let $H$ be a nef divisor on $Z$
that determines one of these contractions. Then \cite{jah-rad-pet}
and \cite[Proposition 1.4.13]{laz} imply that $H$ varies in the
family $H_t$ of nef divisors on $Z_t$.

\begin{prop}
\label{theorem:pic-z-or-rat} $\mathrm{rk}\,\mathrm{Pic}\,Z \ne 2$.
\end{prop}

\begin{proof}
Assume the contrary. It follows from the condition
$\mathrm{rk}\,\mathrm{Pic}^G\,Z = 1$ that both of the extremal
contractions on each $Z_t$ above must be either birational or Mori
fibrations. Now \cite[\S 12.3]{isk-pro} and
Remark~\ref{remark:k-not-16} show that $Z$ can only be a divisor
in $\p^2\times\p^2$ of bidegree $(2,2)$. Let us show that $Z$ is
smooth (this will contradict $|\mathrm{Sing}\,Z| = 10$ and prove
Proposition~\ref{theorem:pic-z-or-rat}).

Let $x_i$ (resp. $y_i$) be coordinates on the first (resp. second)
factor of $\p^2\times\p^2$. Let also $f(x,y) = 0$ be the equation
of $Z$ (so that it defines a conic in $\p^2$ whenever $x :=
[x_0:x_1:x_2]$ or $y$ is fixed).

Note that projections to the $\p^2$-factors induce conic bundle
structures on $Z$. These are interchanged by $G$ (because of
$\mathrm{rk}\,\mathrm{Pic}^G\,Z = 1$) and are
$\left<h,\tau^2\right>$-invariant. One may assume that
$\mathrm{Sing}\,Z$ belongs to the affine chart $x_0=y_0=1$ on
$\p^2\times\p^2$. Then, after a coordinate change, we obtain that
$f(x,y) = x_1x_2y_1y_2 + x_1x_2 + y_1y_2 + 1$ in this chart, with
$h$ acting diagonally on $x_i$ and $y_i$.

Now, differentiating $f(x,y)$ by $x_1,x_2$ we get $x_i = -y_1y_2$,
and similarly $y_i = -x_1x_2$. This gives $x_1 = x_2,y_1 =
y_2\in\{-1,-w\}$, which contradicts $f(x,y) = 0$.
\end{proof}

Note that there is a $G$-invariant surface $S \in |-K_Z|$, since
$\p^8 = \p(\com^9)\supset Z$, similarly as in
Remark~\ref{remark:k-not-16}.

\begin{lemma}
\label{theorem:x-s-plt-1} The pair $(Z,S)$ is plt.
\end{lemma}

\begin{proof}
As in the proof of Lemma~\ref{theorem:x-s-plt}, it suffices to
exclude the case when (the normalization of) the surface $S$ is
ruled, over some base curve $B$ of genus $\le 1$.

Note that any line $L$ passing through two points from
$\text{Sing}\,Z$ is contained in $Z$ (as $Z$ is an intersection of
quadrics). In particular, we have $S \cdot L > 0$ for $>10$ of
such $L$, which yields either $G\subset\text{PGL}(2,\com)$ or a
faithful $G$-action on $B$, a contradiction.
\end{proof}

\begin{prop}
\label{theorem:pic-z-or-rat-1} $\mathrm{rk}\,\mathrm{Pic}\,Z \ne
1$.
\end{prop}

\begin{proof}
Assume the contrary. Then $Z_t\subset\p^8$ from the beginning of
{\ref{subsection:pro-4}} are Fano $3$-folds of the principal
series.

It follows from Lemma~\ref{theorem:x-s-plt-1} that $S$ is normal
and connected. Further, we have $k \le 2$ and $-K_Z\cdot\p^1\le 2$
in \eqref{deg-k-z}, which means (cf. Lemma~\ref{theorem:no-g-p-1})
that the exceptional locus of $\psi: Y\map Z$ consists of two
disjoint surfaces, say $E_1,E_2$, so that $L_i := \psi(E_i)$ are
two lines on $Z$. In particular, there is a $G$-invariant subspace
$\p^3\subset\p^8$, with $Z\cap\p^3 = L_1 \cup L_2$, such that $X$
is obtained from $Z$ via the linear projection from $\p^3$ (recall
that both $X$ and $Z$ are anticanonically embedded).

We may assume that $Z\cap\p^3\subset S$ (otherwise there is a
pencil as in Remark~\ref{remark:k-not-16}). Hence $S$ contains the
$(-2)$-curve $L_1$ (we have identified $S$ with its minimal
resolution). Note that $L_1$ is preserved by the group
$\left<h\right>$.

Consider the quotient $S_h := S\slash\left<h\right>$. Then the
image of $L_1$ on $S_h$ has self-intersection $=-2/5$ by the
projection formula. On the other hand, this self-intersection
$\in\cel[0.5]$ (for $S_h$ has at most canonical singularities due
to $h^*(\omega_{S_h}) = \omega_{S_h}$), a contradiction.

Proposition~\ref{theorem:pic-z-or-rat-1} is completely proved.
\end{proof}

It follows from Propositions~\ref{theorem:pic-z-or-rat} and
\ref{theorem:pic-z-or-rat-1} that $\mathrm{rk}\,\mathrm{Pic}\,Z >
2$. Now, since $-K_Z^3 = 12$, from \cite{namikawa},
\cite{jah-rad-pet} and \cite[\S\S 12.4 -- 12.6]{isk-pro} we obtain
that $Z$ is a deformation of $Z_t = $ either $\p^1\times[\text{del
Pezzo surface of degree 2}]$ or a double cover of
$\p^1\times\p^1\times\p^1$, ramified along a divisor of tridegree
$(2,2,2)$. In both cases, $Z$ is hyperelliptic (cf. the beginning
of the proof of Proposition~\ref{theorem:pic-z-or-rat}), which
contradicts Lemma~\ref{theorem:k-z-is-emb}.

The proof of Theorem~\ref{theorem:main} is finished.

\bigskip

\section{Concluding discussion}
\label{section:con}

\refstepcounter{equation}
\subsection{}
\label{subsection:con-1}

Equations \eqref{s-6-qua} and the results of \cite{cynk} show that
any $S_6$-invariant quartic $X_t$ is not $\ra$-factorial. In turn,
as we saw in Section~\ref{section:class}, it is indispensable to
compute the group $\text{Cl}\,X_t = H_4(X_t,\cel)$ (e.g. for the
arguments of Section~\ref{section:pro} to carry on).

This amazing interrelation between topology and (birational)
geometry of $X_t$ provides one with a hint for studying the
birational type of $X_t$ by ``topological" means. In this regard,
let us give a sketch of an argument, showing that $X_t$ is
unirational for generic $t\in\mathbb{R}$, hence for (again
generic) $t\in\com$ (cf. \cite[Proposition 2.3]{deF}).

Namely, differentiating \eqref{s-6-qua} one interprets this system
of equations as the graph of a \emph{Morse function}
$F:\mathbb{R}\p^4\map\mathbb{R}$, so that
$X_t^{\mathbb{R}}=F^{-1}(t)$ are smooth level sets for
$t\not\in\{\infty,0,10/7,2,4,6\}$, while the rest of
$t\not\in\{0,4\}$ correspond to critical level sets of (maximal)
index $3$ (here $X_t^{\mathbb{R}}$ denotes the real locus of
$X_t$).

We may replace $\mathbb{R}\p^4$ with its universal cover $S^4$.
Then $F$ lifts to a Morse function on $S^4$ and thus all smooth
$X_t^{\mathbb{R}}$ are homotopy $\mathbb{R}\p^3$. Actually generic
$X_t^{\mathbb{R}}$ is \emph{diffeomorphic} to $\mathbb{R}\p^3$
(note that this $X_t^{\mathbb{R}}$ is smooth and connected).

Further, $X_t^{\mathbb{R}}$ is contained in an affine space
$\mathbb{R}^N$, some $N$, because $\sum x_i^4 \ne 0$ over
$\mathbb{R}$. Then the function $F_p := \text{dist}(\cdot,p)$
defines a Morse function on $X_t^{\mathbb{R}}$ for very general
points $p\in\mathbb{R}^N$. (Here $\text{dist}(x,y) := \|x-y\|^2$
is the standard Euclidean distance.)

The layers of $F_p$ yield a vector field on $X_t^{\mathbb{R}}$,
which is non-degenerate and normal to these layers outside two
points, where this field vanishes. We thus obtain a (Hopf)
fibration on $X_t^{\mathbb{R}}$ with a section
$F_p^{-1}(o)\setminus{\{\text{2 points}\ o_1,o_2\}} =
\mathbb{R}\p^2$ such that $F_p^{-1}(o)\subset X_t^{\mathbb{R}}$ as
an algebraic subset. It remains to apply a diffeomorphism over
$F_p^{-1}(o)\setminus{\{o_1,o_2\}}$ which makes
$X_t^{\mathbb{R}}\setminus{\{F_p^{-1}(o_1),F_p^{-1}(o_2)\}} =
\mathbb{R}\p^1\times F_p^{-1}(o)\setminus{\{o_1,o_2\}}$ as
algebraic varieties.

The upshot of the above discussion is that $X_t^{\mathbb{R}}$
(hence $X_t$) admits \emph{many cancellations} in the sense of
\cite{bkk}. This implies that $X_t$ is unirational.

\refstepcounter{equation}
\subsection{}
\label{subsection:con-2}

We conclude with the following questions:

\begin{itemize}

    \item What is the Fano $3$-fold which the quartic $X_6$ is
    $G$-birationally isomorphic to (cf. Section~\ref{section:pro})?

    \smallskip

    \item Are there non-trivial $G$-birational modifications
    of $X_6$ for other subgroups $G \subset S_6$?

    \smallskip

    \item Is $X_t$ unirational over a number field
    field?\footnote{Note that all rational quartics are $\ra$-rational.}

    \smallskip

    \item Does the set of $\mathbb{Q}$-points on $X_t$ satisfy the potential
    density property?

    \smallskip

    \item Does $X_t$ carry a pencil of (birationally)
    Abelian surfaces?\footnote{Again this holds for rational $X_t$.}

\end{itemize}

\end{document}